\documentclass[reqno,a4paper]{amsart}

\usepackage{amsmath,amssymb, amsthm}
\usepackage{dsfont}

\usepackage{color}
\makeatletter
\@addtoreset{equation}{section}
\makeatother
\makeatletter

\DeclareMathOperator{\Real}{Re}
\DeclareMathOperator{\Imm}{Im}

\newcommand{\R}{{\mathds R}}
\newcommand{\Z}{{\mathds Z}}
\newcommand{\z}{{\mathbb{R}^N}}
\newcommand{\C}{{\mathds C}}

\newcommand{\E}{\mathcal{E}}
\newcommand{\A}{\mathfrak{a}}
\newcommand{\cL}{\mathcal{L}}

\newtheorem{definition}{Definition}[section]
\newtheorem{lemma}[definition]{Lemma}
\newtheorem{theorem}[definition]{Theorem}
\newtheorem{proposition}[definition]{Proposition}
\newtheorem{remark}[definition]{Remark}
\newtheorem{corollary}[definition]{Corollary}

\title[Fourth-order Schr\"odinger type operator]{Fourth-order Schr\"odinger type operator with singular potentials}

\author[F. Gregorio]{Federica Gregorio}
\address{Dipartimento di Fisica, Universit\`a degli Studi di Salerno, Via Ponte Don Melillo, 84084 FISCIANO (Sa), Italy.}
\email{fgregorio@unisa.it}

\author[S. Mildner]{Sebastian Mildner}
\address{Technische Universit\"at Dresden,
Institut f\"ur Analysis,
01062 Dresden,
Germany}
\email{sebastian.mildner@tu-dresden.de}

\@namedef{subjclassname@2010}{%
  \textup{2010} Mathematics Subject Classification}

\thanks{The first author is member of the Gruppo Nazionale per l'Analisi Matematica,
la Probabilit\`a e le loro Applicazioni (GNAMPA) of the Istituto Nazionale di Alta Matematica
(INdAM)}



\date{}
\begin{document}

\date{}
 
\begin{abstract}
In this paper, we study the biharmonic operator perturbed by an inverse fourth-order potential. In particular, we consider the operator $A=\Delta^2-c\lvert x\rvert^{-4}$ where $c$ is any constant such that $c<\bigl(\frac{N(N-4)}{4}\bigr)^2$. The semigroup generated by $-A$ in $L^2(\z)$, $N\geq5$, extrapolates to a bounded holomorphic $C_0$-semigroup on $L^p(\z)$ for $p\in [p'_0,p_0]$ where $p_0=\frac{2N}{N-4}$ and $p'_0$ is its dual exponent.  Moreover, we obtain the boundedness of the Riesz transform $\Delta A^{-1/2}$ on $L^p(\z)$ for all $p\in(p'_0,2]$.
\end{abstract}

\maketitle
\section{Introduction}

Let us consider the biharmonic operator
\[A_0=\Delta^2.\]
In this paper, we want to study the perturbation of $A_0$ with the singular potential $V(x)=c\lvert x\rvert^{-4}$ a.e.\ where $c<C^*:=\bigl(\frac{N(N-4)}{4}\bigr)^2$. More precisely, we consider the operator
\[A=A_0-V=\Delta^2-\frac{c}{\lvert x\rvert^4}.\]

The biharmonic operator $A_0$ is included in a class of higher order elliptic operators studied by Davies in \cite{DAVIES}. In particular, he proves that for $N<4$, $(e^{-tA_0})_{t\geq 0}$ induces a bounded $C_0$-semigroup on $L^p(\z)$ for all $1\leq p<\infty$ and that Gaussian-type estimates for the heat kernel hold.
Denoting by $K$ the heat kernel associated to the operator $A_0$, he proves that there exist $c_1, c_2, k>0$ such that
\[\lvert K(t,x,y)\rvert
\leq c_1 t^{-N/4}e^{-c_2\frac{\lvert x-y\rvert^{4/3}}{t^{1/3}}+kt}\]
for all $t>0$ and $x,y\in\z$.

The result is different for $N>4$. In this case he proves that the semigroup $(e^{-tA_0})_{t\geq 0}$ extends to a bounded holomorphic $C_0$-semigroup on $L^p(\z)$ for all $p\in[p'_0,p_0]$, where $p_0=\frac{2N}{N-4}$ and $p'_0$ is its dual exponent. An analogous situation holds when replacing $A_0$ by $A$,
which was remarked for example in \cite[Section 6]{LSV} by Liskevich, Sobol and Vogt (see also Remark \ref{rem:LSV} and Proposition \ref{prop:LSV} below).

We can define the Riesz transform associated to $A$ by
\[\Delta A^{-1/2}
:=\frac{1}{\Gamma(1/2)}\int_0^\infty t^{-1/2}\Delta e^{-tA}\,dt.\]
The boundedness of the Riesz transform on $L^p(\z)$ implies that the domain of $A^{1/2}$ is included in the Sobolev space $W^{2,p}(\z)$. Thus, we obtain $W^{2,p}$-regularity of the solution to the evolution equation with initial datum in $L^p(\z)$. The boundedness of the Riesz transforms for Schr\"odinger operators has widely been studied in harmonic analysis. Several authors have generalized the results for elliptic operators $L$ of order $2m$ or for Riemannian manifolds, see for example \cite{AT}, \cite{AO}, \cite{BK} and the references therein. Blunck and Kunstmann in \cite{BK} apply the Calder\'on-Zygmund theory for non-integral operators to obtain estimates on $\Delta L^{-1/2}$ since, in general, operators of order $2m$ do not satisfy Gaussian bounds if $2m<N$. More precisely, they prove an abstract criterion for estimates of the type
\[\lVert B L^{-\alpha}f\rVert_{L^p(\Omega)}
\leq C_p\lVert f\rVert_{L^p(\Omega)},
\qquad p\in(q_0,2], \]
where $B, L$ are linear operators, $\alpha\in[0,1)$, $q_0\in[1,2)$ and $\Omega$ is a measure space. We will apply this criterion (Theorem \ref{thm:blunckkust} below) to our situation $(B,L,\Omega)=(\Delta, A, \z)$ with $q_0=p'_0$.

We will treat the operator $A=\Delta^2-V$
in $L^2(\z)$ as the associated operator to the form
\[
\A(u,v)
=(\Delta u,\Delta v)_2-(Vu,v)_2
\]
with $D(\A)=\{u\in H^2(\z): \lVert\lvert V\rvert^{1/2}u\rVert_2<\infty\}$. As a consequence of the Rellich inequality, 
\begin{equation}\label{eq:hardyin}
\Bigl(\frac{N(N-4)}{4}\Bigr)^2\int_\z\frac{\lvert u(x)\rvert^2}{\lvert x\rvert^4}\,dx
\leq\int_\z\lvert\Delta u(x)\rvert^2\,dx
\end{equation}
for all $u\in H^2(\z)$ with $N\geq 5$ (cf.\ \cite{TZ}), one obtains $D(\A)=H^2(\z)$ and
\begin{equation}\label{eq:hardyalpha}
\A(u):=\int_\z\lvert\Delta u(x)\rvert^2\,dx-\int_\z V(x)\lvert u(x)\rvert^2\,dx
\geq\eta\int_\z\lvert\Delta u(x)\rvert^2\,dx
\end{equation}
for some $\eta\in(0,1)$, i.e., $\A$ is densely defined and positive semi-definite. Moreover, thanks to \eqref{eq:hardyin} and \eqref{eq:hardyalpha}, the norms $\lVert{}\cdot{}\rVert_{\A}$ and $\lVert{}\cdot{}\rVert_{H^2}$ are equivalent, therefore $(D(\A),\lVert{}\cdot{}\rVert_{\A})$ is complete, i.e., $\A$ is closed. 
Consequently, see for example \cite[Theorem 1.54]{EMO}, $-A$ is the generator of a $C_0$-semigroup $(e^{-tA})_{t\geq 0}$ on $L^2(\z)$
that is contractive and holomorphic on the sector $\Sigma(\pi/2)$.

In the following, making use of multiplication operators and off-diagonal estimates, we prove that, for $N\geq5$, the semigroup $(e^{-tA})_{t\geq 0}$ extrapolates to a bounded holomorphic $C_0$-semigroup on $L^p(\z)$ for all $p\in[p'_0,p_0]$ and that the Riesz transform associated to $A$ is bounded on $L^p(\z)$ for all $p\in(p'_0,2]$.

\medskip
\emph{Notation.} Throughout the paper, we assume $N\geq 5$ and denote by $\lVert{}\cdot{}\rVert_{p\to q}$ the norm of operators acting from $L^p(\z)$ into $L^q(\z)$ and $p'$ is the dual exponent of $p$, $p'=\frac{p}{p-1}$. Further, we set $\Sigma(\theta):=\{z\in\C\setminus\{0\}:\lvert\arg z\rvert<\theta\}$ for $\theta\in(0,\pi/2]$. Finally, $\chi_E$ denotes the characteristic function of a set $E$ and $B(x,r)$ the ball around $x$ of radius $r$.

\section{The twisted semigroup}

In order to show the boundedness of the Riesz transform and obtain off-diagonal estimates for the semigroup generated by $-A$ we use the classical Davies perturbation technique, and estimate the twisted semigroup. Therefore, denoting by $\alpha$ a multi-index with $\lvert\alpha\rvert=\alpha_1+\dots+\alpha_N$ and $D^\alpha$ the corresponding partial differential operator on $C^\infty(\z)$, we define $\E:=\{\phi\in C^\infty(\z;\R)\textrm{ bounded}:\lvert D^\alpha\phi\rvert\leq1\ \textrm{for all}\ 1\leq\lvert\alpha\rvert\leq 2\}$ and the twisted forms
\[\A_{\lambda\phi}(u,v):=\A(e^{-\lambda\phi}u,e^{\lambda\phi}v)\]
with $D(\A_{\lambda\phi})=H^2(\z)$, $\lambda\in\R$ and $\phi\in\E$.
A simple computation shows that
\[A_{\lambda\phi}:=e^{\lambda\phi}Ae^{-\lambda\phi}\]
with $D(A_{\lambda\phi})=\{u\in L^2(\z): e^{-\lambda\phi}u\in D(A)\}$ is the associated operator to the form $\A_{\lambda\phi}$.
Moreover, there exist $0<\gamma<1$ and $k>1$ such that the inequality
\begin{equation}\label{eq:forme}
\lvert\A_{\lambda\phi}(u)-\A(u)\rvert
\leq\gamma \A(u)+k(1+\lambda^4)\lVert u\rVert_2^2
\end{equation}
holds for all $u\in H^2(\z)$, $\lambda\in\R$ and $\phi\in\E$. Indeed,
we have
\begin{align*}
\A_{\lambda\phi}(u)
&= \A(u)+\lambda^4\int_\z\lvert\nabla\phi\rvert^4\lvert u\rvert^2\,dx
-\lambda^2\int_\z\lvert\Delta\phi\rvert^2\lvert u\rvert^2\,dx\\
&\quad
+4\lambda^3 i \Imm\int_\z\lvert\nabla\phi\rvert^2\nabla\phi\cdot\nabla\overline{u}\,u\,dx
+2\lambda^2\Real\int_\z\lvert\nabla \phi\rvert^2u\Delta \overline{u}\,dx\\
&\quad
-4\lambda^2\Real\int_\z\Delta\phi\,\nabla\phi\cdot\nabla\overline{u}\,u\,dx
+2\lambda i \Imm\int_\z\Delta\phi\,\overline{u}\Delta u\,dx\\
&\quad
-4\lambda^2\int_\z\lvert\nabla\phi\cdot\nabla u\rvert^2\,dx
+4\lambda i \Imm\int_\z\nabla\phi\cdot\nabla \overline{u}\Delta u\,dx.
\end{align*}
Now, the application of \eqref{eq:hardyalpha}, the Gagliardo-Nirenberg inequality
\[
 \lVert\nabla u\rVert_2^2
\leq\lVert u\rVert_2\lVert\Delta u\rVert_2,
\qquad u\in H^2(\z)
\]
and Young's inequality yields for $0<\varepsilon<1$
\begin{align*}
\lvert\A_{\lambda\phi}(u)-\A(u)\rvert
&\leq N^2(\lambda^4+\lambda^2)\lVert u\rVert_2^2
\\
&\quad+4(N\lambda^2\lVert u\rVert_2)
(N^{1/2}\lvert\lambda\rvert\lVert\nabla u\rVert_2)
+2(N\lambda^2\varepsilon^{-1}\lVert u\rVert_2)
(\varepsilon\lVert\Delta u\rVert_2)
\\
&\quad+4(N\lvert\lambda\rvert\lVert u\rVert_2)
(N^{1/2}\lvert\lambda\rvert\lVert\nabla u\rVert_2)
+2(N\lvert\lambda\rvert\varepsilon^{-1}\lVert u\rVert_2)
(\varepsilon\lVert\Delta u\rVert_2)
\\
&\quad+4N\lambda^2\lVert\nabla u\rVert_2^2
+4(N^{1/2}\lvert\lambda\rvert\varepsilon^{-1}\lVert\nabla u\rVert_2)(\varepsilon\lVert\Delta u\rVert_2)
\\
&\leq 4\varepsilon^2\lVert\Delta u\rVert_2^2
+4N^2(\lambda^2+\lambda^4)\varepsilon^{-2}\lVert u\rVert_2^2
+10N\lambda^2\varepsilon^{-2}\lVert\nabla u\rVert_2^2
\\
&\leq 4\varepsilon^2\lVert\Delta u\rVert_2^2
+4N^2(\lambda^2+\lambda^4)\varepsilon^{-2}\lVert u\rVert_2^2
\\
&\quad+10(N\lambda^2\varepsilon^{-3}\lVert u\rVert_2)
(\varepsilon\lVert\Delta u\rVert_2)
\\
&\leq 9\varepsilon^2\lVert\Delta u\rVert_2^2
+9N^2(\lambda^2+\lambda^4)\varepsilon^{-6}\lVert u\rVert_2^2
\\
&\leq (9\varepsilon^2/\eta)\A(u)
+18N^2(1+\lambda^4)\varepsilon^{-6}\lVert u\rVert_2^2.
\end{align*}

For the rest of this article, we fix $\gamma$ and $k$ such that inequality \eqref{eq:forme} holds. Then the forms $\A_{\lambda\phi}+2k(1+\lambda^4)$ are closed and uniformly sectorial (see for example \cite[Theorem 1.19]{EMO}). Thus the operators $-A_{\lambda\phi}-2k(1+\lambda^4)$ generate contractive holomorphic $C_0$-semigroups on $L^2(\z)$ with a common sector of holomorphy $\Sigma(\Theta)$.
Therewith, we can show the following lemma.

\begin{lemma}\label{lem:ssd}
 \textup{(a)}
For all $z\in\Sigma(\Theta)$, $\lambda\in\R$ and $\phi\in\E$ the following inequality holds
\begin{equation}\label{eq:22esttw}
 \lVert e^{-zA_{\lambda\phi}}\rVert_{2\to2}\leq e^{2k(1+\lambda^4)\Real z}.
\end{equation}

 \textup{(b)}
There exists $M_\Theta>0$ such that
\begin{equation}\label{eq:twistedlapl}
\lVert \Delta e^{-zA_{\lambda\phi}}\rVert_{2\to2}
\leq M_\Theta \lvert z\rvert^{-1/2}e^{2k(1+\lambda^4)\Real z}
\end{equation}
holds for all  $z\in\Sigma(\Theta/2)$, $\lambda\in\R$ and $\phi\in\E$.
\end{lemma}

\begin{proof}
Let $\lambda\in\R$ and $\phi\in\E$.
As mentioned before, we have
\begin{equation}\label{eq:chsg1}
\lVert e^{-z(A_{\lambda\phi}+2k(1+\lambda^4))}\rVert_{2\to 2}
\leq 1,
\end{equation}
for all $z\in\Sigma(\Theta)$,
which implies \eqref{eq:22esttw}. Moreover, by the Cauchy formula,
\begin{equation}\label{eq:chsg2}
\lVert (A_{\lambda\phi}+2k(1+\lambda^4)) e^{-z(A_{\lambda\phi}+2k(1+\lambda^4))}\rVert_{2\to 2}
\leq (\lvert z\rvert\sin(\Theta/4))^{-1}
\end{equation}
holds for all $z\in\Sigma(\Theta/2)$.
Further, \eqref{eq:hardyalpha} and \eqref{eq:forme} yield
\begin{align*}
 (1-\gamma)\eta\lVert\Delta v\rVert_2^2
&\leq (1-\gamma)\A(v)
\leq \Real(\A_{\lambda\phi}(v)+2k(1+\lambda^4)\lVert v\rVert_2^2)\\
&\leq\lVert (A_{\lambda\phi}+2k(1+\lambda^4))v\rVert_2\lVert v\rVert_2
\end{align*}
for all $v\in D(A_{\lambda\phi})$.
Taking $v=e^{-z(A_{\lambda\phi}+2k(1+\lambda^4))}u$ and applying the estimates \eqref{eq:chsg1} and \eqref{eq:chsg2}, we conclude \eqref{eq:twistedlapl} with $M_\Theta=1/\sqrt{(1-\gamma)\eta\sin(\Theta/4)}$.
\end{proof}

Finally, we prove $L^p-L^q$ estimates for the twisted semigroups.

\begin{lemma}\label{lem:pdest}
Let $p_0'\leq p\leq 2\leq q\leq p_0$. Then there exists $M_{pq}>0$ such that
\[
\lVert e^{-zA_{\lambda\phi}} u\rVert_q
\leq M_{pq}\lvert z\rvert^{-\frac{N}{4}(\frac{1}{p}-\frac{1}{q})}e^{2k(1+\lambda^4)\Real z}\lVert u\rVert_p
\]
holds for all $z\in\Sigma(\Theta/2)$, $\lambda\in\R$, $\phi\in\E$ and $u\in L^2(\z)\cap L^p(\z)$.
\end{lemma}
\begin{proof}
Let $z\in\Sigma(\Theta/2)$, $\lambda\in\R$ and $\phi\in\E$. Then, by Sobolev's embedding theorem (cf. \cite[Theorem 4.31]{ADAMS}) and Lemma \ref{lem:ssd}, one obtains
\begin{equation}\label{eq:2p0est}
\lVert e^{-zA_{\lambda\phi}}u\rVert_{\frac{2N}{N-4}}
\leq C_{\mathrm{S}}\lVert \Delta e^{-zA_{\lambda\phi}}u\rVert_2
\leq C_{\mathrm{S}}M_\Theta\lvert z\rvert^{-1/2}e^{2k(1+\lambda^4)\Real z}\lVert u\rVert_2
\end{equation}
for all $u\in L^2(\z)$.
Applying the Riesz-Thorin interpolation theorem to $e^{-zA_{\lambda\phi}}$ with respect to the bounds \eqref{eq:22esttw} and \eqref{eq:2p0est},
we achieve the $L^2-L^q$ estimate 
\[
\lVert e^{-zA_{\lambda\phi}}\rVert_{2\to q}
\leq M_{2q}\lvert z\rvert^{-\frac{N}{4}(\frac{1}{2}-\frac{1}{q})}e^{2k(1+\lambda^4)\Real z}
\]
with $M_{2q}=(C_\mathrm{S}M_\Theta)^{\frac{N}{2}(\frac{1}{2}-\frac{1}{q})}$. Then a duality argument yields the $L^p-L^2$ estimate.
Finally, we only have to combine these two and use the semigroup property to conclude the $L^p-L^q$ estimate with $M_{pq}=(2C_\mathrm{S}M_\Theta)^{\frac{N}{2}(\frac{1}{p}-\frac{1}{q})}$.
\end{proof}

\section{Off-diagonal estimates}

In this section, we study off-diagonal estimates, which enable us to obtain the extrapolation of the semigroup $(e^{-tA})_{t\geq 0}$ and the boundedness of the Riesz transform $\Delta A^{-1/2}$.

We say that a family $(T(z))_{z\in\Sigma(\theta)}$, $\theta\in(0,\pi/2]$, of bounded linear operators on $L^2(\z)$ \emph{satisfies $L^p-L^q$ off-diagonal estimates} for $1\leq p\leq q\leq \infty$ if there exist $c_1,c_2>0$ such that for each convex, compact subsets $E,F$ of $\z$, for each $u\in L^2(\z)\cap L^p(\z)$ supported in $E$ and for all $z\in\Sigma(\theta)$, such an inequality holds
\[
\lVert T(z)u\rVert_{L^q(F)}
\leq c_1\lvert z\rvert^{-\gamma_{pq}}\exp\Bigl(-c_2\frac{d(E,F)^{4/3}}{\lvert z\rvert^{1/3}}\Bigr)\lVert u\rVert_p,
\]
where $\gamma_{pq}=\frac{N}{4}\bigl(\frac{1}{p}-\frac{1}{q}\bigr)$ and
\[
d(E,F)=\sup_{\phi\in\E}[\inf\{\phi(x)-\phi(y):x\in E, y\in F\}].
\]

Davies proved that this distance is equivalent to the Euclidean one if the sets $E$ and $F$ are additionally disjoint, \cite[Lemma 4]{DAVIES}. We recall this result.
\begin{lemma}\label{lem:equiv}
If $E$ and $F$ are disjoint, convex, compact subsets of $\z$, then
\[d_e(E,F)\leq d(E,F)\leq N^{1/2}d_e(E,F),\]
where $d_e(E,F)$ is the Euclidean distance between $E$ and $F$.
\end{lemma}

\begin{remark}\label{rem:dist}
 \textup{(a)}
Since the distance $d$ between non-disjoint sets is zero, we can drop the assumption of disjointedness in the previous lemma without changing the statement.

 \textup{(b)}
For $E,F\subset\z$ compact, convex, $x,y\in\z$ and $r>0$ such that $E\subset B(x,r)$ and $F\subset B(y,r)$ we obtain
\[
 d(E,F)^{4/3}
\geq 2^{-1/3}\lvert x-y\rvert^{4/3}-(2r)^{4/3}.
\]
Indeed, we can estimate as follows
\[
 \lvert x-y\rvert
\leq 2r+d_e(E,F)
\leq 2r+d(E,F)
\leq 2^{1/4}((2r)^{4/3}+d(E,F)^{4/3})^{3/4}.
\]
\end{remark}

The following proposition relates the results of the previous section with the notion of off-diagonal estimates.
\begin{proposition}\label{prop:offdiag}
 Let $\theta\in(0,\pi/2]$ and $(T(z))_{z\in\Sigma(\theta)}$ be a family in $\cL(L^2(\z))$ that satisfies
\begin{equation}\label{eq:sp}
 T(z)=D_sT(s^4z)D_{1/s},
\qquad s\in(0,1),\ z\in\Sigma(\theta),
\end{equation}
where $D_s$ is the dilation operator, i.e., $D_sv(x)=v(sx)$ a.e. for all $v\in L^1_{\mathrm{loc}}(\z)$.
Further let $1\leq p\leq q<\infty$ and $M,\omega>0$ such that 
\[
 \lVert e^{\lambda\phi}T(z)e^{-\lambda\phi}u\rVert_q
\leq M\lvert z\rvert^{-\gamma_{pq}}e^{\omega(1+\lambda^4)\lvert z\rvert}\lVert u\rVert_p
\]
holds for all $z\in\Sigma(\theta)$, $\lambda>0$, $\phi\in\E$ and $u\in L^2(\z)\cap L^p(\z)$.
Then $(T(z))_{z\in\Sigma(\theta)}$ satisfies $L^p-L^q$ off-diagonal estimates.
\end{proposition}

\begin{proof}
 Let $z\in\Sigma(\theta)$, $E,F$ be convex, compact subsets of $\z$
and $u\in L^2(\z)\cap L^p(\z)$ supported in $E$.
Then the assumption yields 
\begin{align*}
 \lVert T(z)u\rVert_{L^q(F)}
&\leq\lVert e^{-\lambda\phi}\chi_F\rVert_\infty \lVert e^{\lambda\phi}T(z)e^{-\lambda\phi}e^{\lambda\phi}\chi_E u\rVert_q\\
&\leq e^{-\lambda\inf_F\phi}
M\lvert z\rvert^{-\gamma_{pq}}e^{\omega(1+\lambda^4)\lvert z\rvert}\lVert\chi_E e^{\lambda\phi}\rVert_\infty\lVert u\rVert_p\\
&\leq e^{-\lambda(\inf_F\phi-\sup_E\phi)}
M\lvert z\rvert^{-\gamma_{pq}}e^{\omega(1+\lambda^4)\lvert z\rvert}\lVert u\rVert_p
\end{align*}
for all $\lambda>0$ and $\phi\in\E$.
Minimising the right-hand side with respect to $\phi\in\E$
and choosing $\lambda$ as $\bigl(\frac{d(E,F)}{4\omega\lvert z\rvert}\bigr)^{1/3}$
we obtain
\[
\lVert T(z)u\rVert_{L^q(F)}
\leq e^{2\omega|z|}
M\lvert z\rvert^{-\gamma_{pq}}\exp\Bigl(-c_\omega\frac{d(E,F)^{4/3}}{\lvert z\rvert^{1/3}}\Bigr)\lVert u\rVert_p
\]
with $c_\omega=\frac{3}{4(4\omega)^{1/3}}$.
Now, we use the scaling property to get rid of the factor $e^{2\omega\lvert z\rvert}$. For $s\in(0,1)$ we estimate
\begin{align*}
\lVert T(z)u\rVert_{L^q(F)}
&=\lVert D_s\chi_{sF}T(s^4z)\chi_{sE}D_{1/s}u\rVert_q\\
&=s^{-\frac{N}{q}}\lVert \chi_{sF}T(s^4z)\chi_{sE}D_{1/s}u\rVert_q\\
&\leq e^{2\omega s^4\lvert z\rvert}M\lvert z\rvert^{-\gamma_{pq}}
\exp\Bigl(-c_\omega\frac{(d(sE,sF)/s)^{4/3}}{\lvert z\rvert^{1/3}}\Bigr)
s^{-\frac{N}{p}}\lVert D_{1/s}u\rVert_p\\
&\leq e^{2\omega s^4\lvert z\rvert}M\lvert z\rvert^{-\gamma_{pq}}\exp\Bigl(-\frac{c_\omega}{N^{2/3}}\frac{d(E,F)^{4/3}}{\lvert z\rvert^{1/3}}\Bigr)\lVert u\rVert_p.
\end{align*}
Taking $s\to 0$, we get $L^p-L^q$ off-diagonal estimates for $(T(z))_{z\in\Sigma(\theta)}$.
\end{proof}

Now, since $(e^{-zA})_{z\in\Sigma(\Theta/2)}$ satisfies the scaling property \eqref{eq:sp} thanks to the invariance of the Laplacian, we can infer from Lemma \ref{lem:pdest} the following statement.
\begin{corollary}
 $(e^{-zA})_{z\in\Sigma(\Theta/2)}$ satisfies $L^p-L^q$ off-diagonal estimates for all $p\in[p_0',2]$ and $q\in[2,p_0]$.
\end{corollary}

Finally, we are able to state the following theorem.

\begin{theorem}
The semigroup $(e^{-tA})_{t\geq 0}$ on $L^2(\z)$ extrapolates to a bounded holomorphic $C_0$-semigroup on $L^p(\z)$ for all $p\in[p'_0,p_0]$.
\end{theorem}

\begin{proof}
It suffices to show that the family $(e^{-zA})_{z\in\Sigma(\Theta/2)}$ is uniformly bounded on $L^p(\z)$ to infer the extrapolation to a bounded holomorphic $C_0$-semigroup on $L^p(\z)$. Moreover, we only have to treat the case $p\in(2,p_0]$. 

Let $p\in(2,p_0]$, $z\in\Sigma(\Theta/2)$
and $C_n$ be the cube with centre $n\lvert z\rvert^{1/4}$ and edge length $\lvert z\rvert^{1/4}$ for all $n\in\Z^N$.
Then, using the $L^2-L^p$ off-diagonal estimates for $(e^{-zA})_{z\in\Sigma(\Theta/2)}$, Remark \ref{rem:dist}(b) and H\"older's inequality, we obtain
\begin{align*}
\lVert\chi_{C_n} e^{-zA}\chi_{C_m} u\rVert_p
&\leq c_1e^{-c_2\lvert m-n\rvert^{4/3}}\lvert z\rvert^{-\gamma_{2p}}\lvert C_m\rvert^{\frac{1}{2}-\frac{1}{p}}\lVert\chi_{C_m} u\rVert_p\\
&= c_1e^{-c_2\lvert m-n\rvert^{4/3}}\lVert\chi_{C_m} u\rVert_p
\end{align*}
for all $m,n\in\Z^N$ and $u\in L^2(\z)\cap L^p(\z)$ with $c_1,c_2>0$ independent of $z$, $u$, $m$ and $n$.
Since the operator $B\colon\ell^1(\Z^N)\to\ell^1(\Z^N)$ with
\[
 (Bx)_m=c_1\sum_{n\in\Z^N}e^{-c_2\lvert m-n\rvert^{4/3}}x_n,
\qquad m\in\Z^N,\ x\in\ell^1(\Z^N)
\]
is bounded on $\ell^1(\Z^N)$ as well as on $\ell^\infty(\Z^N)$,
the Riesz-Thorin interpolation theorem yields that $B$ is also bounded on $\ell^p(\Z^N)$.
Setting $\hat{u}=(\lVert\chi_{C_n}u\rVert_p)_{n\in\Z^N}$
we conclude
\[
\Vert e^{-zA}u\Vert_p
\leq\lVert B\hat{u}\rVert_{\ell^p}
\leq\lVert B\rVert_{\ell^p\to\ell^p}\lVert\hat{u}\rVert_{\ell^p}
\leq\lVert B\rVert_{\ell^p\to\ell^p}\lVert u\rVert_{p}
\]
for all $u\in L^2(\z)\cap L^p(\z)$.
\end{proof}

\begin{remark}\label{rem:LSV}
We have provided this proof as an application of the previous results, which we will also need in the next section to prove the boundedness of the Riesz transform. Actually, we could have also applied \cite[Proposition 6.1]{LSV}, which holds in a general setting of higher order operators defined by closed, sectorial sesquilinear forms. We recall this statement according to the notations of our situation.
\end{remark}

\begin{proposition}\label{prop:LSV}
Let $\A$ be a closed, sectorial sesquilinear form in $L^2(\z)$ with $D(\A)=H^2(\z)$ such that for some $C,k>0$ 
\[
\tfrac{1}{2}\lVert \Delta u\rVert_2^2
\leq\Real \A(u)
\leq C(\lVert \Delta u\rVert_2^2+\lVert u\rVert_2^2)
\]
and
\[
\lvert\A_{\lambda\phi}(u)-\Real \A(u)\rvert
\leq\tfrac{1}{4}\Real\A(u)+k(1+\lambda^4)\lVert u\rVert_2^2
\]
hold for all $u\in H^2(\z)$, $\lambda\geq 0$ and $\phi\in\E$.
Then the holomorphic $C_0$-semigroup $(e^{-tA})_{t\geq 0}$ on $L^2(\z)$, associated with $\A$, extrapolates to a holomorphic $C_0$-semigroup $T_p=(e^{-tA_p})_{t\geq 0}$ on $L^p(\z)$ for all $p\in[p'_0,p_0]$. The sector of holomorphy of $T_p$ and the spectrum $\sigma(A_p)$ are $p$-independent.
\end{proposition}

\section{Riesz transform}

We show that $\Delta A^{-1/2}\in\mathcal{L}(L^p(\z))$ for all $p\in(\frac{2N}{N+4},2]$. We already know that the Riesz transform of the operator $A$ is bounded on $L^2(\z)$ thanks to the inequality
\[
\eta\lVert\Delta u\rVert_2^2
\leq a(u)
=\lVert A^{1/2}u\rVert_2^2,
\qquad u\in H^2(\R^N)
\]
and the selfadjointness of $A^{1/2}$. Then, provided $\Delta A^{-1/2}$ is of weak type $(p'_0,p'_0)$, we can use the Marcinkiewicz interpolation theorem to obtain the boundedness on $L^p(\z)$ for $p'_0< p\leq2$.

Let us recall the definition of weak type operators. Let $(X,\Sigma,\mu)$ be a measure space. An operator $L$ is of weak type $(p,p)$ for $1\leq p<\infty$, if there exists a constant $C$ such that for any measurable function $f$, such a relation holds
\[\mu\{x:\lvert Lf(x)\rvert\geq\lambda\}
\leq C\lambda^{-p}\lVert f\rVert_p^p.\]

In order to prove that $\Delta A^{-1/2}$ is of weak type $(p_0',p_0')$ we make use of \cite[Theorem 1.1]{BK} in the following adapted form.

\begin{theorem}\label{thm:blunckkust}
Let $1\leq p<2<q\leq\infty$, $q_0\in(p,\infty]$ and $(e^{-tA})_{t\geq 0}$ be a bounded holomorphic semigroup on $L^2(\z)$ such that $A$ is injective and has dense range. Further, let $\alpha\in[0,1)$ and $B$ a linear operator satisfying $D(A^\alpha)\subset D(B)$ and the weighted norm estimates
\begin{align}\label{eq:riesz1}
\lVert \chi_{B(x,t^{1/4})}e^{-tA}\chi_{B(y,t^{1/4})}\rVert_{p\to q}
&\leq c_1t^{-\gamma_{pq}}\exp\Bigl(-c_2\frac{\lvert x-y\rvert^{4/3}}{t^{1/3}}\Bigr)\\
\label{eq:riesz2}
\lVert \chi_{B(x,t^{1/4})}t^{\alpha}Be^{-e^{i\sigma}tA}\chi_{B(y,t^{1/4})}\rVert_{p\to q_0}
&\leq c_1t^{-\gamma_{pq_0}}\exp\Bigl(-c_2\frac{\lvert x-y\rvert^{4/3}}{t^{1/3}}\Bigr)
\end{align}
hold for all $x,y\in\z$, $t>0$, $\lvert\sigma\rvert<\frac{\pi}{2}-\theta$, for some $\theta>0$. Then $BA^{-\alpha}$ is of weak type $(p,p)$ provided $BA^{-\alpha}$ is of weak type $(2,2)$.
\end{theorem}

We can now state the main result of this section.

\begin{theorem}\label{thm:boundriesz}
The Riesz transform of the operator $A$ is bounded on $L^p(\z)$ for all $p\in(p_0',2]$.
\end{theorem}

\begin{proof}
We will show that the assumptions of Theorem \ref{thm:blunckkust} in the setting $(B,\alpha,p,q,q_0)=(\Delta,1/2,p_0',p_0,2)$ are satisfied to infer that $\Delta A^{-1/2}$ is of weak type $(p'_0,p'_0)$.

First, we observe that $A$ is injective and selfadjoint and has therefore dense range. Moreover, we have $D(A^{1/2})=D(\Delta)$ and $\Delta A^{-1/2}$ is bounded on $L^2(\z)$, hence of weak type $(2,2)$, as was pointed out above.
Now, it remains to show that estimates of the form \eqref{eq:riesz1} and \eqref{eq:riesz2} are satisfied. Due to Remark \ref{rem:dist}(b), such estimates are direct consequences of $L^{p_0'}-L^{p_0}$ off-diagonal estimates for $(e^{-zA})_{z\in\Sigma(\Theta/2)}$, which we have already obtained, and $L^{p_0'}-L^2$ off-diagonal estimates for the family $(\lvert z\rvert^{1/2}\Delta e^{-zA})_{z\in\Sigma(\Theta/2)}$. To achieve the last ones, we show that
\begin{equation}\label{eq:2perdelta}
\lVert e^{\lambda\phi}\Delta e^{-zA}e^{-\lambda\phi}\rVert_{2\to2}
\leq M\lvert z\rvert^{-1/2}e^{\omega(1+\lambda^4)\lvert z\rvert}
\end{equation}
holds for all $z\in\Sigma(\Theta/2)$, $\lambda\in\R$ and $\phi\in\E$ with some $M,\omega>0$. Indeed, we compute
\begin{align*}
e^{\lambda\phi}\Delta e^{-zA}e^{-\lambda\phi}u
&= e^{\lambda\phi}\Delta e^{-\lambda\phi}e^{-zA_{\lambda\phi}}u\\
&=(\lambda^2\lvert\nabla\phi\rvert^2-\lambda\Delta\phi)e^{-zA_{\lambda\phi}}u
-2\lambda\nabla\phi\cdot\nabla e^{-zA_{\lambda\phi}}u\\
&\quad +\Delta e^{-zA_{\lambda\phi}}u,
\end{align*}
which can be estimated, thanks to \eqref{eq:22esttw} and \eqref{eq:twistedlapl}, in the following way
\begin{align*}
\lVert e^{\lambda\phi}\Delta e^{-zA}e^{-\lambda\phi}u\rVert_2^2
&\leq 8N^2(1+\lambda^4)\lVert e^{-zA_{\lambda\phi}}u\rVert_2^2\\
&\quad +16N\lambda^2\lVert\nabla e^{-zA_{\lambda\phi}}u\rVert_2^2
+4\lVert\Delta e^{-zA_{\lambda\phi}}u\rVert_2^2\\
&\leq 16N^2(1+\lambda^4)\lVert e^{-zA_{\lambda\phi}}u\rVert_2^2
+12\lVert\Delta e^{-zA_{\lambda\phi}}u\rVert_2^2\\
&\leq 16N^2(1+\lambda^4)e^{4k(1+\lambda^4)\lvert z\rvert}\lVert u\rVert_2^2\\
&\quad+12M_\Theta^2\lvert z\rvert^{-1}e^{4k(1+\lambda^4)\lvert z\rvert}\lVert u\rVert_2^2\\
&\leq 16(M_\Theta^2+N^2)\lvert z\rvert^{-1}e^{5k(1+\lambda^4)\lvert z\rvert}\lVert u\rVert_2^2.
\end{align*}
Combining inequality \eqref{eq:2perdelta} with the $L^{p_0'}-L^2$ estimate of Lemma \ref{lem:pdest}, we get
\[
 \lVert e^{\lambda\phi}\Delta e^{-zA}e^{-\lambda\phi}\rVert_{p_0'\to 2}
\leq 2M_{p_0'2}M\lvert z\rvert^{-1}e^{\omega(1+\lambda^4)\lvert z\rvert}
\]
for all $z\in\Sigma(\Theta/2)$, $\lambda>0$ and $\phi\in\E$.
Since the family $(\lvert z\rvert^{1/2}\Delta e^{-zA})_{z\in\Sigma(\Theta/2)}$ satisfies the scaling property \eqref{eq:sp}, it also satisfies $L^{p_0'}-L^2$ off-diagonal estimates by Propostion \ref{prop:offdiag}.

 Thus, $\Delta A^{-1/2}$ is of weak type $(p'_0,p'_0)$. Now, by the boundedness of $\Delta A^{-1/2}$ on $L^2(\z)$ and the Marcinkiewicz interpolation theorem, we conclude that  $\Delta A^{-1/2}\in\mathcal{L}(L^p(\z))$ for all $p\in(p'_0,2]$.
\end{proof}

Finally, we obtain the following corollary.
\begin{corollary}
The parabolic problem associated to $-A=-\Delta^2+\frac{c}{\lvert x\rvert^4}$, $c<C^*$
\[
\left\{\begin{aligned}
\partial_tu(t)&=-Au(t)\qquad\textrm{for}\ t\geq0,\\
u(0)&=f,
\end{aligned}
\right.
\]
admits a unique solution for each initial datum $f\in L^p(\z)$, $p\in[p'_0, p_0]$. Moreover, if $f\in L^p(\z)$ for $p\in(p'_0,2]$, then  the solution is in $W^{2,p}(\z)$.
\end{corollary}

\section*{Acknowledgements}
One of the authors, F. G., would like to thank Professor El Maati Ouhabaz and Professor Abdelaziz Rhandi for their very helpful suggestions and discussions. She is also grateful to the Institut de Math\'ematiques of the University of Bordeaux for the kind hospitality during her visit.

\end{document}